\documentclass[11pt]{article}

\usepackage{amssymb,amsmath}
\usepackage{amsthm}
\usepackage{hyperref}
\usepackage{mathrsfs}
\usepackage{textcomp}
\hypersetup{
    colorlinks,%
    citecolor=black,%
    filecolor=black,%
    linkcolor=black,%
    urlcolor=black
}

\usepackage{verbatim}

\usepackage{sectsty}

\makeatletter

\newdimen\bibspace
\renewenvironment{thebibliography}[1]{%
 \section*{\refname 
       \@mkboth{\MakeUppercase\refname}{\MakeUppercase\refname}}%
     \list{\@biblabel{\@arabic\c@enumiv}}%
          {\settowidth\labelwidth{\@biblabel{#1}}%
           \leftmargin\labelwidth
           \advance\leftmargin\labelsep
           \itemsep\bibspace
           \parsep\z@skip     %
           \@openbib@code
           \usecounter{enumiv}%
           \let\p@enumiv\@empty
           \renewcommand\theenumiv{\@arabic\c@enumiv}}%
     \sloppy\clubpenalty4000\widowpenalty4000%
     \sfcode`\.\@m}
    {\def\@noitemerr
      {\@latex@warning{Empty `thebibliography' environment}}%
     \endlist}

\makeatother

\makeatletter

\newtheorem{thm}{Theorem}[section]
\newtheorem{lem}[thm]{Lemma}
\newtheorem{prop}[thm]{Proposition}

\newtheorem{cor}[thm]{Corollary}

\def\Xint#1{\mathchoice
  {\XXint\displaystyle\textstyle{#1}}%
  {\XXint\textstyle\scriptstyle{#1}}%
  {\XXint\scriptstyle\scriptscriptstyle{#1}}%
  {\XXint\scriptscriptstyle\scriptscriptstyle{#1}}%
  \!\int}
\def\XXint#1#2#3{{\setbox0=\hbox{$#1{#2#3}{\int}$}
  \vcenter{\hbox{$#2#3$}}\kern-.5\wd0}}

\def\dashint{\Xint-}

                \newcommand{\lda}{\lambda}
                
\newcommand{\va}{\varepsilon}           \newcommand{\ud}{\mathrm{d}}
\newcommand{\be}{\begin{equation}}      \newcommand{\ee}{\end{equation}}

\newcommand{\R}{\mathbb{R}}

\begin{document}
\title{\bf\Large The local behavior of positive solutions  for higher order equation with isolated singularity }
\vspace{0.25cm}
\author{\medskip {{ Yimei Li }
}\\
\small Department of Mathematic, School of Sciences,  Beijing Jiaotong University, \\
\small Beijing 100044, People's Republic of China}
\date{}
\maketitle

\vspace{-0.8cm}

{\bf Abstract}
We use blow up analysis for local integral equations to provide a blow up rates of solutions of higher order Hardy-H\'{e}non equation in a bounded domain with an isolated singularity, and show the asymptotic radial symmetry of the solutions near the singularity. This work generalizes the  correspondence results  of Jin-Xiong \cite{JX} on higher order conformally invariant equations with an isolated singularity.



\section{Introduction}
This article aims to study  the local  behaviors of positive  solutions for the higher order  Hardy-H\'{e}non equation
\be\label{Weq:maineq1}
(-\Delta)^\sigma u=|x|^{\tau}u^p \quad\quad\mbox{in }\ \ B_1 \backslash\{0\},
\ee
where   $1\leq\sigma<\frac{n}{2}$ is  an integer,  $\tau>-2\sigma$, $p>1$ and the punctured unit ball $B_1\backslash\{0\}\subset\R^n$,  $n\ge 2$.

In the special case of $\sigma=1$, the local behavior of the positive solutions for \eqref{Weq:maineq1} with isolated singularity has been very well understood.  For $\tau>-2$, $1<p\leq\frac{n+2}{n-2}$, the blow up rate of the solution 
\[
u(x)\leq C|x|^{-\frac{2+\tau}{p-1}},  \ \ |\nabla u(x)|\leq C|x|^{-\frac{p+1+\tau}{p-1}}\quad \mbox{near }\ x=0,
\]
is obtained by a number of authors, where $\nabla u$ denotes the gradient of $u$ and $C$ stands for the different positive constants.  For more precise estimates and details, we refer the interested reader to  \cite{AA, A, GS, KMPS, Lions,  N,PSP, ZZ}. In the classical paper \cite{CGS}, Caffarelli-Gidas-Spruck established the asymptotic behavior for local positive solutions of \eqref{Weq:maineq1}, 
\[
u(x)=\bar u(|x|)(1+O(|x|))\quad\quad\mbox{as }\ x\to 0,
\]
where $\tau=0$, $\frac{n}{n-2}\leq p \leq \frac{n+2}{n-2}$ and $\bar u(|x|):=\dashint_{\mathbb{S}^n}u(|x|\theta)\ud \theta$ is the spherical average of $u$. Li  \cite{L} improved  their results for  $\tau\leq 0$, $1<p\leq \frac{n+2+\tau}{n-2}$,  and simplified the proofs. For the fractional case $0<\sigma<1$,  Caffarelli-Jin-Sire-Xiong \cite{CJSX} studied the sharp blow up rate, asymptotically radially symmetric and removability of the positive solution for the  fractional Yamabe equation with an isolated singularity
\[
(-\Delta)^\sigma u=u^{\frac{n+2\sigma}{n-2\sigma}} \quad\quad\mbox{in }\ \ B_1 \backslash\{0\}.
\]
Motivated by the work of the above, we  have studied the fractional Hardy-H\'{e}non equations   in our previous work \cite{BL}, and not only derived that there exists a positive constant $C$ such that  the blow up rates
\[
u(x)\leq C|x|^{-\frac{2\sigma+\tau}{p-1}},  \ \ |\nabla u(x)|\leq C|x|^{-\frac{2\sigma+\tau+p-1}{p-1}}\quad \mbox{near }\ x=0,
\]
for $\tau>-2\sigma$, $1<p<\frac{n+2\sigma}{n-2\sigma}$, but also obtained the asymptotically radially symmetric
\[
u(x)=\bar u(|x|)(1+O(|x|))\quad \mbox{as }\ x\to 0,
\]
for $-2\sigma<\tau\leq0$, $\frac{n+\tau}{n-2\sigma}<p\leq \frac{n+2\sigma+2\tau}{n-2\sigma}$, which is consistent with the classic case $\sigma=1$.

Recently, by using blow up analysis Jin-Xiong \cite{JX}   proved sharp blow up rates of the positive solutions of higher order conformally invariant equations  with an isolated singularity
\[
(-\Delta)^\sigma u=u^{\frac{n+2\sigma}{n-2\sigma}} \quad\quad\mbox{in }\ \ B_1 \backslash\{0\},
\]
where $1\leq\sigma<\frac{n}{2}$ is  an integer, and showed the asymptotic radial symmetry of
the solutions near the singularity. In detail, they proved that  there exists a positive constant $C$ such that 
\[
u(x)\leq C|x|^{-\frac{n-2\sigma}{2}}\quad \mbox{near }\ x=0,
\]
and
\[
u(x)=\bar u(|x|)(1+O(|x|))\quad \mbox{as }\ x\to 0.
\]
This is an extension of the celebrated theorem of Caffarelli-Gidas-Spruck \cite{CGS} for the second order Yamabe equation
and Caffarelli-Jin-Sire-Xiong \cite{CJSX} for the fractional Yamabe equation with isolated singularity to higher order equations.

Inspired by the above work,  we are interested in the higher order Hardy-H\'{e}non  equation \eqref{Weq:maineq1}, where $1\leq\sigma<\frac{n}{2}$ is  an integer, in a bounded domain with an isolated singularity in this paper. Our  results provide a  blow up rate estimate near the isolated singularity and show that the solution of \eqref{Weq:maineq1} is asymptotically radially symmetric near the isolated singularity, which is consistent with the case $0<\sigma\leq 1$.
\begin{thm}\label{Wthm:a}Suppose that $1\leq\sigma<\frac{n}{2}$ is  an integer, and $u\in C^{2\sigma}(B_1 \backslash\{0\})\cap L^{\frac{n}{n-2\sigma}}(B_{1})$ is a positive solution of \eqref{Weq:maineq1}.

(i) If  $-2\sigma<\tau$, $\frac{n+\tau}{n-2\sigma}<p<\frac{n+2\sigma}{n-2\sigma}$ and
\be\label{fuhao}
(-\Delta)^m u\geq 0\quad\quad\mbox{\rm{in} }\ \ B_1 \backslash\{0\},\quad m=1,2,\cdots,\sigma-1,
\ee
then there exists a positive constant $C=C(n,\sigma, \tau, p,)$ such that 
\[
u(x)\leq C|x|^{-\frac{2\sigma+\tau}{p-1}},\ \ \ |\nabla u(x)|\leq  C|x|^{-\frac{2\sigma+\tau+p-1}{p-1}}\quad\text{\rm{near} } \ x=0.
\]

(ii) If $-2\sigma<\tau\leq0$, $\frac{n+\tau}{n-2\sigma}<p\leq\frac{n+2\sigma+2\tau}{n-2\sigma}$ and the solution satisfies \eqref{fuhao}, then
\[
u(x)=\bar u(|x|)(1+O(|x|))\quad\quad{\rm{as}}\ x\to 0,
\]
where $\bar u(|x|):=\dashint_{\mathbb{S}^n}u(|x|\theta)\ud \theta$ is the spherical average of $u$.
\end{thm}

The main idea of our approach  is to carry out blow up analysis  to get the blow up rate estimate near the isolated singularity, and by the method of moving spheres to study the asymptotically radially symmetric as in Caffarelli-Jin-Sire-Xiong \cite{CJSX}  for the fractional Yamabe equation $0<\sigma<1$.  The method of moving spheres (see \cite{Li06, LLin, LZhang, LZhu}) has become a very powerful tool in the study of nonlinear elliptic equations, i.e. the method of moving planes together with the conformal invariance, which fully exploits the conformal invariance of the problem. It is known that one of the conformal invariance, i.e. the Kelvin transform of $u$ defined as
\[
u_{x,\lda }(y):= \left(\frac{\lda }{|y-x|}\right)^{n-2\sigma}u\left(x+\frac{\lda^2(y-x)}{|y-x|^2}\right)\quad\quad\mbox{in }\ \ \R^{n},
\]
with $\lda>0$ and $x\in \R^n$, plays an important part in our proof. However, in our local situation \eqref{Weq:maineq1}, the sign conditions \eqref{fuhao} may change when performing the Kelvin transforms. Inspired by a unified approach
to solve the Nirenberg problem and its generalizations  by the authors  Jin-Li-Xiong in \cite{JLX}, we shall make use of integral representations. In details, we first  prove $|x|^{\tau}u^p\in L^1(B_1)$ under the assumptions of Theorem \ref{Wthm:a}, and then we can rewrite the differential equation \eqref{Weq:maineq1} into the integral equation involving the Riesz potential
\[
u(x)=\int_{B_1}\frac{|y|^{\tau}u^p(y)}{|x-y|^{n-2\sigma}}dy+h(x)\quad\quad\mbox{in }\ \ B_1 \backslash\{0\},
\]
where  $h\in C^1(B_1)$ is a positive function. Thus, the sign conditions \eqref{fuhao} will ensure the maximum principle and are essential for applying the moving spheres method. As a result, we just need to study  the integral equation.

This paper is organized as follows. In Section \ref{sec:follow}, we shall show that \eqref{Weq:maineq1} can be written as the form of \eqref{i1}, and then give some results about the integral equation, which implies that Theorem \ref{Wthm:a} follows from these results. In Section \ref{sec:bd}, we prove the upper bound near the isolated singularity for the solution of \eqref{i1}, and   the asymptotic radial symmetry will be obtained in Section \ref{sec:asymptotical}.

\section{Proof of the main results} \label{sec:follow}
Before that we suppose $0<\sigma<\frac{n}{2}$  is a real number, $-2\sigma<\tau$, $p>1$, $u\in C(B_1\backslash\{0\})$,  $|x|^{\tau}u^p(x)\in L^1(B_1)$, and we consider the integral equation involving the Riesz potential
\be\label{i1}
u(x)=\int_{B_1}\frac{|y|^{\tau}u^p(y)}{|x-y|^{n-2\sigma}}dy+h(x)\quad\quad\mbox{in }\ \ B_1 \backslash\{0\},
\ee
where $h\in C^1(\overline{B_1})$ is a positive function, otherwise we  consider the equation in a smaller ball. About the integral equation \eqref{i1}, we shall first show some results,  which will  recover our previous work \cite{BL} for the fractional Yamabe equation $0<\sigma<1$,  and the proof will be given later in Section \ref{sec:bd} and Section \ref{sec:asymptotical}.  Now we first introduce the upper bound of the positive solution near the singularity.
\begin{thm}\label{i2}
Suppose that  $0<\sigma<\frac{n}{2}$  is a real number, $-2\sigma<\tau$, $1<p<\frac{n+2\sigma}{n-2\sigma}$, if $u\in C(B_1\backslash\{0\})$ is a positive solution of \eqref{i1} and $|x|^{\tau}u^p(x)\in L^1(B_1)$, then there exists a positive constant $C=C(n,\sigma, \tau, p)$ such that
\be\label{Weq:low and up bound}
u(x)\leq C|x|^{-\frac{2\sigma+\tau}{p-1}},\ \ \ |\nabla u(x)|\leq  C|x|^{-\frac{2\sigma+\tau+p-1}{p-1}}\quad\text{\rm{near} } \ x=0.
\ee
\end{thm}
One consequence of the upper bound of the solution near the singularity in Theorem \ref{i2} is the following Harnack inequality.
\begin{cor}\label{j1}
Assume as in Theorem \ref{i2}, then for all $0<r<\frac{1}{4}$, then there exists a positive constant  $C$ independent of $r$ such that
\[
\sup_{B_{3r/2}\backslash B_{r/2}}u\leq C\inf_{B_{3r/2}\backslash B_{r/2}}u.
\]
\end{cor}
The following theorem shows the asymptotic radial symmetry of the positive solution near the singularity.
\begin{thm}\label{i3}
Suppose that $0<\sigma<\frac{n}{2}$  is a real number, $-2\sigma<\tau\leq0$, $\frac{n+\tau}{n-2\sigma}<p\leq\frac{n+2\sigma+2\tau}{n-2\sigma}$, if $u\in C(B_1\backslash\{0\})$ is a positive solution of \eqref{i1} and $|x|^{\tau}u^p(x)\in L^1(B_1)$,  then
\[
u(x)=\bar u(|x|)(1+O(|x|))\quad\quad{\rm{as}}\ x\to 0,
\]
where $\bar u(|x|):=\dashint_{\mathbb{S}^n}u(|x|\theta)\ud \theta$ is the spherical average of $u$.
\end{thm}

\subsection{Proof of Theorem \ref{Wthm:a}}
Next we shall show that we can rewrite the differential equation \eqref{Weq:maineq1} into the integral equation \eqref{i1} involving the Riesz potential, more precise, if $u\in C^{2\sigma}(B_1\backslash\{0\})\cap L^{\frac{n}{n-2\sigma}}(B_{1})$ is a positive solution of \eqref{Weq:maineq1},   then 
\be\label{ii4}
u(x)=B(n,\sigma)\int_{B_r}\frac{|y|^{\tau}u^p(y)}{|x-y|^{n-2\sigma}}dy+h_1(x)\quad\quad\mbox{in }\ \ B_1 \backslash\{0\},
\ee
with
\[
B(n,\sigma):=\frac{\Gamma\left(\frac{n-2\sigma}{2}\right)}{2^{2\sigma}\pi^{n/2}\Gamma(\sigma)},
\]
where $\Gamma$ is the Gamma function, and $h_1$ is smooth in $B_r$ and satisfies $(-\Delta)^{\sigma}h_1=0$ in $B_r$. As a result, we can finish the proof of Theorem \ref{Wthm:a} by Theorem \ref{i2} and Theorem \ref{i3}.  For the purpose,  we first need the following proposition.

\begin{prop}\label{j2}
Suppose that $1\leq \sigma <\frac{n}{2}$ is an integer, $\tau>-2\sigma$,  $p>\frac{n+\tau}{n-2\sigma}$, and $u\in C^{2\sigma}(B_1\backslash\{0\})$ is a positive solution of \eqref{Weq:maineq1}, then $|x|^{\tau}u^p\in L^1(B_1)$.
\begin{proof}
 To do so, we take a smooth function $\eta$ defined in $\R$ as the cut-off function with values in $[0,1]$ satisfying
 \[
\eta(t):=
\begin{cases}
\begin{aligned}
&0,&\quad&\mbox{\rm if }\  |t|\leq 1,\\
&1,&\quad&\mbox{\rm if }\  |t|\geq 2.
\end{aligned}
\end{cases}
\]
For small $\varepsilon>0$, let $\varphi_{\varepsilon}(x)=\eta(\varepsilon^{-1}|x|)^q$ with $q=\frac{2\sigma p}{p-1}$. Multiplying both sides by $\varphi_{\varepsilon}(x)$ and using integration by parts, we have
\begin{equation*}
\begin{split}
\int_{B_1}|x|^{\tau}u^p \varphi_{\varepsilon} &=\int_{B_1}u(-\Delta)^\sigma\varphi_{\varepsilon}+
\int_{\partial B_1}\frac{\partial (-\Delta)^{\sigma-1}u }{\partial \nu}ds\\
&\leq C\varepsilon^{-2\sigma}\int_{\varepsilon\leq|x|\leq 2\varepsilon}u\eta(\varepsilon^{-1}|x|)^{q-2\sigma}+C\\
&=C\varepsilon^{-2\sigma}\int_{\varepsilon\leq|x|\leq 2\varepsilon}u\varphi_{\varepsilon}^{\frac{1}{p}}+C\\
&=C\varepsilon^{-2\sigma}\int_{\varepsilon\leq|x|\leq 2\varepsilon}|x|^{\frac{\tau}{p}}u\varphi_{\varepsilon}^{\frac{1}{p}}|x|^{-\frac{\tau}{p}}+C\\
&\leq C\varepsilon^{-2\sigma-\frac{\tau}{p}}\int_{\varepsilon\leq|x|\leq 2\varepsilon}|x|^{\frac{\tau}{p}}u\varphi_{\varepsilon}^{\frac{1}{p}}+C\\
&\leq C\varepsilon^{-2\sigma-\frac{\tau}{p}+n-\frac{n}{p}}\left(\int_{B_1}|x|^{\tau}u^p\varphi_{\varepsilon}\right)^{\frac{1}{p}}+C,
\end{split}
\end{equation*}
where  the H\"{o}lder inequality is used in the above inequality. Since $p>\frac{n+\tau}{n-2\sigma}$, we conclude that 
\[
\int_{2\varepsilon\leq|x|\leq 1}|x|^{\tau}u^p <\int_{B_1}|x|^{\tau}u^p \varphi_{\varepsilon} \leq C.
\]
By sending $\varepsilon\rightarrow 0$, we obtain
\[
\int_{B_1}|x|^{\tau}u^p \leq C.
\]
Thus, we obtain  that $|x|^{\tau}u^p\in L^1(B_1)$ and complete the proof.
\end{proof}
\end{prop}

Furthermore, we also need to  recall  some known facts.  Let $G_1(x,y)$ be the Green function of  $-\Delta$ on the unit ball, i.e.
\[
G_1(x,y)=\frac{1}{(n-2)w_{n-1}}\left(|x-y|^{2-n}-\left|\frac{x}{|x|}-|x|y\right|^{2-n}\right)\quad\mbox{for} \ \ x,y\in B_1,
\]
and define
\[
G_{\sigma}(x,y):=\int_{B_1\times\cdots\times B_1}G_1(x,y_1)G_1(y_1,y_2)\cdots G_1(y_{\sigma-1},y)dy_1\cdots dy_{\sigma-1},
\]
then we  have
\[
G_{\sigma}(x,y)=B(n,\sigma)|x-y|^{2\sigma-n}+A_{\sigma}(x,y),
\]
where  $A_{\sigma}(\cdot,\cdot)$ is smooth in $B_1\times B_1$.  Let
\[
H_1(x,y):=-\frac{\partial}{\partial\nu_y}G_1(x,y)=\frac{1-|x|^2}{w_{n-1}|x-y|^{n}}\quad\mbox{for} \ \ x\in B_1,\ y\in \partial B_1,
\]
where $w_{n-1}$ is the surface area of the unit sphere in $\R^n$, then for $2\leq i\leq \sigma$, define
\[
H_{i}(x,y):=\int_{B_1\times\cdots\times B_1}G_1(x,y_1)G_1(y_1,y_2)\cdots G_1(y_{i-2},y_{i-1})H_1(y_{i-1},y)dy_1\cdots dy_{i-1}.
\]
Furthermore, for a function $u\in C^{2}(B_{1})\cap C(\overline{B_1})$,  we have
\[
u(x)=\int_{B_{1}}G_1(x,y)(-\Delta u)(y)dy+\int_{\partial B_{1}}H_1(x,y)(-\Delta u)(y)dS_{y}.
\]
By induction,  we have for $2\sigma<n$, $u\in C^{2\sigma}(B_{1})\cap C^{2\sigma-2}(\overline{B_1})$,  we have
\[
u(x)=\int_{B_{1}}G_\sigma(x,y)(-\Delta)^{\sigma} u(y)dy+\sum_{i=1}^{\sigma}\int_{\partial B_{1}}H_i(x,y)(-\Delta)^{i-1} u(y)dS_{y}.
\]

Now, we start our proof of Theorem \ref{Wthm:a} by using the above argument

\begin{proof}[Proof of Theorem \ref{Wthm:a}]
We can suppose that $u\in C^{2\sigma}(\overline{B_1}\backslash\{0\})$ and $u>0$ in $\overline{B_1}$, otherwise we just consider the equation in a smaller ball.  By the above argument, we know that  we only need to obtain \eqref{ii4}, then we can finish the proof. To prove \eqref{ii4}, let
\[
v(x):=\int_{B_1}G_{\sigma}(x,y)|y|^{\tau}u^p(y)dy+\sum_{i=1}^{m}\int_{\partial B_1}H_{i}(x,y)(-\Delta)^{\sigma-i}u(y)dS_y,
\]
and
\[
w:=u-v.
\]
Then
\[
(-\Delta)^{\sigma}w=0\quad\quad\mbox{in} \ \ B_1\backslash\{0\}.
\]
By the generalized  Bocher's Theorem \cite{FKM} for polyharmonic function, 
\[
w(x)=\sum_{|\alpha|\leq 2\sigma-1}A_{\alpha}D^{\alpha}(|x|^{2\sigma-n})+g(x),
\]
where $\alpha=(\alpha_{1}, \alpha_{2}, \cdots, \alpha_{n})$ is multi-index,  $A_{\alpha}$ are constants, and $g(x)$ is a smooth
solution of $(-\Delta)^\sigma g(x)=0$ in $B_1$. If we can claim that $A_{\alpha}=0$ for $|\alpha|\leq 2\sigma-1$, then $w(x)$ is a classical polyharmonic function on $B_1$, that is,
\[
(-\Delta)^\sigma w(x)=0\quad\quad\mbox{in} \ \ B_1. 
\]
Moreover, since $w=\Delta w=\cdots=\Delta^{\sigma-1}w=0$ on $\partial B_1$, $w=0$ which implies that $u=v$.  Thus,
\[
u(x)=B(n,\sigma)\int_{B_r}\frac{|y|^{\tau}u^p(y)}{|x-y|^{n-2\sigma}}dy+h_1(x),
\]
where
\begin{equation*}
\begin{split}
h_1(x)=&\int_{B_r}A_{\sigma}(x,y)|y|^{\tau}u^p(y)dy+\int_{B_1\backslash B_r}G_{\sigma}(x,y)|y|^{\tau}u^p(y)dy\\
&+\sum_{i=1}^\sigma\int_{\partial B_1}H_i(x,y)(-\Delta)^{i-1}u(y)dS_y.
\end{split}
\end{equation*}
Since $-\Delta u\geq 0$ in $B_1\backslash\{0\}$, and $u>0$ in $\overline{B_1}$, we know from the Maximum Principle that $c_1:=
\inf_{B_1}u=\min_{\partial B_1}u>0$. By $|y|^{\tau}u^p(y)\in L^1(B_1)$ from Proposition \ref{j2}, we can find that $r<\frac{1}{4}$ such that for $x\in B_r$,
\[
\int_{B_r}|A_{\sigma}(x,y)||y|^{\tau}u^p(y)dy\leq\frac{c_1}{2}.
\]
Hence, by condition \eqref{fuhao}, we have for $x\in B_r$,
\begin{equation*}
\begin{split}
h_1(x)&\geq -\frac{c_1}{2}+\int_{\partial B_1}H_i(x,y)u(y)dS_y\\
&\geq -\frac{c_1}{2}+\inf_{B_1}u=\frac{c_1}{2}.
\end{split}
\end{equation*}
On the other hand, $h_1$ is smooth in $B_r$ and satisfies $(-\Delta)^{\sigma}h_1=0$ in $B_r$. Then we can finish the proof.

To do it,  by contradiction, we may assume that there exists a multi-index  $\alpha_{0}\in \R^{n}$ satisfying $|\alpha_{0}|\leq 2\sigma-1$ such that $A_{\alpha_{0}}\neq0$. Thus, for large $\lambda$, we
infer
\be\label{w2}
\left|\left\{x\in B_\rho:|w(x)|>\lambda\right\}\right|> C \lambda^{-\frac{n}{n-2\sigma}}.
\ee
On the other hand, combining with $|y|^{\tau}u^p(y)\in L^1(B_1)$ and the fact that the Riesz potential $|y|^{2\sigma-n}$ is weak type $\left(1,\frac{n}{n-2\sigma}\right)$, then we obtain that $v\in L_{weak}^{\frac{n}{n-2\sigma}}(B_1)\cap L^1(B_1)$. Moreover, for every $\varepsilon>0$ we can choose $\rho>0$ such that $\int_{B_{2\rho}}|y|^{\tau}u^p(y)dy<\varepsilon$, then for all sufficiently large $\lambda$, we have
\[
\left\{x\in B_\rho:|v(x)|>\lambda\right\}\subset \left\{x\in B_\rho:\int_{B_{2\rho}}G_{\sigma}(x,y)|y|^{\tau}u^p(y)dy>\frac{\lambda}{2}\right\}
\] 
which implies that
\[
\left|\left\{x\in B_\rho:|v(x)|>\lambda\right\}\right|\leq \left|\left\{x\in B_\rho:\int_{B_{2\rho}}G_{\sigma}(x,y)|y|^{\tau}u^p(y)dy>\frac{\lambda}{2}\right\}\right|\leq C(n,\sigma)\varepsilon \lambda^{-\frac{n}{n-2\sigma}}.
\]
Due to $u\in L^{\frac{n}{n-2\sigma}}(B_1)$,  we can choose suitable $\rho>0$ such that $\int_{B_\rho}u^{\frac{n}{n-2\sigma}}< \varepsilon$, 
\[
\left|\left\{x\in B_\rho:|u(x)|>\frac{\lambda}{2}\right\}\right|\leq \left(\frac{2}{\lambda}\right)^{\frac{n}{n-2\sigma}}\int_{B_\rho}u^{\frac{n}{n-2\sigma}}\leq 2^{\frac{n}{n-2\sigma}}\varepsilon \lambda^{-\frac{n}{n-2\sigma}}.
\]
Hence, $w\in L_{weak}^{\frac{n}{n-2\sigma}}(B_1)\cap L^1(B_1)$ and for every $\varepsilon>0$, there exist $\rho>0$ such that for all sufficiently large $\lambda$,
\[
\left|\left\{x\in B_\rho:|w(x)|>\lambda\right\}\right|\leq \left|\left\{x\in B_\rho:|u(x)|>\frac{\lambda}{2}\right\}\right|+\left|\left\{x\in B_\rho:|v(x)|>\frac{\lambda}{2}\right\}\right|.
\]
It follows that
\[
\left|\left\{x\in B_\rho:|w(x)|>\lambda\right\}\right|\leq C(n,\sigma)\varepsilon \lambda^{-\frac{n}{n-2\sigma}}.
\]
This is a contradiction with \eqref{w2} provided that $\varepsilon$ is small enough. Up to now, we complete the proof.
\end{proof}

\section{The upper bound near the  isolated singularity} \label{sec:bd}
In this section, we shall give  proofs of Theorem \ref{i2} and Corollary \ref{j1} respectively. The following we  start our proof.
\subsection{Proof of Theorem \ref{i2}}
First, we recall the Doubling Property  \cite[Lemma 5.1]{PQS} and denote $B_{R}(x)$ as the ball in $\R^{n}$ with radius $R$ and center $x$. For convenience, we write $B_R(0)$ as $B_R$ for short.
\begin{prop}\label{prop:double}
Suppose that $\emptyset\neq D\subset \Sigma \subset \R^n$, $\Sigma$ is closed and $\Gamma=\Sigma \setminus D$. Let $M: D\rightarrow (0,\infty)$ be bounded on compact subset of $D$. If  for a  fixed  positive constant $k$, there exists  $y\in D$ satisfying
\[
M(y){\rm{dist}}(y,\Gamma)>2k,
\]
then there exists $x\in D$ such that
\[
M(x)\geq M(y),\quad\quad M(x){\rm{dist}}(x,\Gamma)>2k,
\]
and for all $z\in D\cap B_{kM^{-1}(x)}(x)$,
\[
M(z)\leq 2M(x).
\]
\end{prop}

Next, in order to prove Theorem \ref{i2}, we start with the following lemma.
\begin{lem}\label{Wlemma:a}
Let $1< p <\frac{n+2\sigma}{n-2\sigma}$, $0<\alpha\leq1$ and $c(x)\in C^{2\sigma,\alpha}(\overline{B_1})$ satisfy
\be\label{7}
\|c\|_{C^{2,\alpha}(\overline{B_1})}\leq C_1,\quad c(x)\geq C_2\quad{\rm in }\ \  \overline{B_1}
\ee
for some positive constants $C_1$, $C_2$. Suppose that $h\in C^{1}(B_1)$ and $u\in C^{2\sigma}(B_1)$ is a nonnegative solution of
\be\label{HH}
u(x)=\int_{B_1}\frac{c(y)u^p(y)}{|x-y|^{n-2\sigma}}dy+h(x)\quad\quad {\rm in }\ \  \ B_1,
\ee
then there exists a positive constant  $C$ depending only on $n$, $\sigma$, $p$, $C_1$, $C_2$  such that
\[
|u(x)|^{\frac{p-1}{2\sigma}}+ |\nabla u(x)|^{\frac{p-1}{p+2\sigma-1}}\leq C[{\rm{dist}}(x,\partial B_1)]^{-1}\quad\quad {\rm in}\ \ \ B_{1}.
\]
\end{lem}
\begin{proof}
Arguing by contradiction, we assume that for $k=1,2,\cdots$, there exist nonnegative functions $u_k$ satisfying \eqref{HH} and points $y_k\in B_{1}$ such that
\be\label{H}
|u_k(y_k)|^{\frac{p-1}{2\sigma}}+ |\nabla u_k(y_k)|^{\frac{p-1}{p+2\sigma-1}}>2k[{\rm{dist}}(y_k,\partial B_1)]^{-1}.
\ee
Define
\[
M_k(x):=|u_k(x)|^{\frac{p-1}{2\sigma}}+ |\nabla u_k(x)|^{\frac{p-1}{p+2\sigma-1}}.
\]
Via Proposition \ref{prop:double}, for $D=B_1$, $\Gamma=\partial B_1$, there exists $x_k\in B_{1}$ such that
\begin{equation}\label{Wlemma33}
M_k(x_k)\geq M_k(y_k),\quad M_k(x_k)> 2k[{\rm{dist}}(x_k,\partial B_1)]^{-1}\geq 2k,
\end{equation}
and for any $z\in B_1$ and $|z-x_k|\leq kM_k^{-1}(x_k)$,
\begin{equation}\label{Wlemma3}
M_k(z)\leq 2M_k(x_k).
\end{equation}
It follows from \eqref{Wlemma33} that
\begin{equation}\label{Wlemma4}
\lambda_k:=M^{-1}_k(x_k)\rightarrow 0\quad\quad{\rm{as} }\ k \rightarrow \infty,
\end{equation}
\begin{equation}\label{Wjuli}
\mbox{dist}(x_k,\partial B_1)>2k\lambda_k,\quad\quad\quad{\rm{for} }\ k=1,2,\cdots.
\end{equation}
Consider
\[
w_k(y):=\lambda_{k}^{\frac{2\sigma}{p-1}}u_{k}(x_k+\lambda_{k}y),\ \ v_k(y):=\lambda_{k}^{\frac{2\sigma}{p-1}}h_{k}(x_k+\lambda_{k}y)\quad\quad \mbox{in }\  B_k.
\]
Combining \eqref{Wjuli}, we obtain that for any $y\in B_k$,
\[
|x_k+\lambda_{k}y-x_k|\leq \lambda_{k} |y|\leq \lambda_{k} k <\frac{1}{2}\mbox{dist}(x_k,\partial B_1),
\]
that is,
\[
x_k+\lambda_{k}y\in B_{\frac{1}{2}\mbox{dist}(x_k,\partial B_1)}(x_k)\subset B_1.
\]
Therefore,  $w_k$ is well defined in $B_k$ and
\[
|w_k(y)|^{\frac{p-1}{2\sigma}}=\lambda_{k}|u_{k}(x_k+\lambda_{k}y)|^{\frac{p-1}{2\sigma}},
\]
\[
|\nabla w_k(y)|^{\frac{p-1}{2\sigma+p-1}}=\lambda_{k}|\nabla u_{k}(x_k+\lambda_{k}y)|^{\frac{p-1}{2\sigma+p-1}}.
\]
From \eqref{Wlemma3}, we find that for all  $y\in B_k$,
\[
|u_{k}(x_k+\lambda_{k}y)|^{\frac{p-1}{2\sigma}}+|\nabla u_{k}(x_k+\lambda_{k}y)|^{\frac{p-1}{2\sigma+p-1}}\leq 2\left(|u_k(x_k)|^{\frac{p-1}{2\sigma}}+ |\nabla u_k(x_k)|^{\frac{p-1}{p+2\sigma-1}}\right).
\]
That is,
\be\label{GUJI}
|w_k(y)|^{\frac{p-1}{2\sigma}}+|\nabla w_k(y)|^{\frac{p-1}{2\sigma+p-1}}\leq2\lambda_{k}M_k(x_k)=2.
\ee
Moreover, $w_k$ satisfies
\be\label{ZHUYAO}
w_k(x)=\int_{B_k}\frac{c_k(y)w_k^p(y)}{|x-y|^{n-2\sigma}}dy+v_k(x)\quad\quad \mbox{in }\  B_k,
\ee
and $$|w_k(0)|^{\frac{p-1}{2\sigma}}+|\nabla w_k(0)|^{\frac{p-1}{2\sigma+p-1}}=1,$$ where $c_k(y):=c( x_k+\lambda_{k}y)$. By \eqref{Wlemma4} it follows that
\[
\|v_k\|_{C^1(B_k)}\rightarrow 0.
\]

By condition \eqref{7}, we obtain that $\{c_k\}$ is uniformly bounded in $\R^n$. For each $R>0$, and for all $y$, $z\in B_R$, we have
\[
|D^{\beta}c_k(y)-D^{\beta}c_k(z)|\leq C_1\lambda_k^{|\beta|}|\lambda_k(y-z)|^{\alpha}\leq C_1|y-z|^{\alpha},\quad |\beta|=0,1,\cdots,2\sigma
\]
for $k$  is large enough. Therefore, by Arzela-Ascoli's Theorem, there exists a function $c\in C^{2\sigma}(\R^n)$,  after extracting a subsequence, $c_k\rightarrow c$ in $C^{2\sigma}_{\rm loc}(\R^n)$. Moreover, by \eqref{Wlemma4}, we obtain
\be\label{Chang}
|c_k(y)-c_k(z)|\rightarrow 0\quad\quad\mbox{as }\ k\rightarrow\infty.
\ee
This implies that the function $c$ actually is a constant $C$. By  \eqref{7} again, $c_k \geq C_2>0$, we conclude that $C$ is a positive constant.

On the other hand, applying  the regularity results in Section 2.1 of \cite{JLX}, after passing to a subsequence, we have, for some nonnegative function
$w\in  C^{2,\alpha}_{{\rm loc}}(\R^n)$,
\[
w_k\rightarrow w\quad\quad\mbox{in }\ C^\alpha_{{\rm loc}}(\R^n)
\]
for some $\alpha>0$. Moreover, $w$ satisfies
\be\label{X}
w(x)=\int_{\R^n}\frac{Cw^p(y)}{|x-y|^{n-2\sigma}}dy\quad\quad\mbox{in }\ \R^{n}
\ee
and $$|w(0)|^{\frac{p-1}{2\sigma}}+|\nabla w(0)|^{\frac{p-1}{2\sigma+p-1}}=1.$$ Since $p<\frac{n+2\sigma}{n-2\sigma}$, this contradicts the Liouville-type result \cite[Theorem 1.4]{WX} that the only nonnegative entire solution of \eqref{X} is $w=0$. Then we conclude the lemma.
\end{proof}

We now turn to  prove  Theorem \ref{i2}.
\begin{proof}[Proof of Theorem \ref{i2}]For  $x_0\in B_{1/2}\backslash\{0\}$, we denote $R:=\frac{1}{2}|x_0|$. Then for any $y\in B_1$, we have $\frac{|x_0|}{2}<|x_0+Ry|<\frac{3|x_0|}{2}$, and deduce that $x_0+Ry\in B_1\backslash
\{0\}$.  Define
\[
w(y):=R^{\frac{2\sigma+\tau}{p-1}}u(x_0+Ry),\ \ v(y):=R^{\frac{2\sigma+\tau}{p-1}}h(x_0+Ry).
\]
Therefore, we obtain that
\[
w(x)=\int_{B_1}\frac{c(y)w^p(y)}{|x-y|^{n-2\sigma}}dy+v(x)\quad\quad\mbox{in } \  B_1,
\]
where $c(y):=|y+\frac{x_0}{R}|^{\tau}$. Notice that
$$1<\left|y+\frac{x_0}{R}\right|<3\quad \mbox{in }\ \overline{B_1}.$$ Moreover,
$$\|c\|_{C^{3}(\overline{B_1})}\leq C,\quad c(y)\geq 3^{-2\sigma}\quad \mbox{in }\ \overline{B_1}.$$ Applying Lemma \ref{Wlemma:a}, we obtain that
\[
|w(0)|^{\frac{p-1}{2\sigma}}+ |\nabla w(0)|^{\frac{p-1}{p+2\sigma-1}}\leq C.
\]
That is,
\[
(R^{\frac{2\sigma+\tau}{p-1}}u(x_0))^{\frac{p-1}{2\sigma}}+(R^{\frac{2\sigma+\tau}{p-1}+1}|\nabla u(x_0)|)^{\frac{p-1}{p+2\sigma-1}}\leq C.
\]
Hence,
\[
\begin{array}{ll}
u(x_0)\leq CR^{-\frac{2\sigma+\tau}{p-1}}\leq C|x_0|^{-\frac{2\sigma+\tau}{p-1}},\vspace{0.2cm}\\
|\nabla u(x_0)|\leq CR^{-\frac{2\sigma+\tau+p-1}{p-1}}\leq C|x_0|^{-\frac{2\sigma+\tau+p-1}{p-1}}.
\end{array}
\]
Then Theorem \ref{i2} is proved by the fact that $x_0\in B_{1/2}\setminus\{0\}$ is arbitrary.
\end{proof}
\subsection{Proof of Corollary \ref{j1}}
Using the upper bound, we shall prove the Harnack inequality.
\begin{proof}[Proof of Corollary \ref{j1}]
Let
\[
w(y):=r^{\frac{2\sigma+\tau}{p-1}}u(ry),\ \ v(y):=r^{\frac{2\sigma+\tau}{p-1}}h(ry).
\]
Then
\[
w(x)=\int_{B_{1/r}}\frac{|y|^{\tau}w^p(y)}{|x-y|^{n-2\sigma}}dy+v(x)\quad\quad\mbox{in } \  B_{1/r}\backslash\{0\}.
\]
Theorem \ref{i2} gives that there exists a positive constant $C$ such that
\[
w(x)\leq C\quad\quad\mbox{in} \ \ B_2\backslash B_{1/10}.
\]
For $z\in \partial B_1$, let
\[
g(x)=\int_{B_{1/r}\backslash B_{9/10}(z)}\frac{|y|^{\tau}w^p(y)}{|x-y|^{n-2\sigma}}dy.
\]
For $x_1$, $x_2\in B_{1/2}(z)$,
\begin{equation*}
\begin{split}
g(x_1)
&=\int_{B_{1/r}\backslash B_{9/10}(z)}\frac{|y|^{\tau}w^p(y)}{|x_1-y|^{n-2\sigma}}dy\\
&=\int_{B_{1/r}\backslash B_{9/10}(z)}\frac{|x_2-y|^{n-2\sigma}}{|x_1-y|^{n-2\sigma}}\frac{|y|^{\tau}w^p(y)}{|x_2-y|^{n-2\sigma}}dy\\
&\leq \left(\frac{7}{2}\right)^{n-2\sigma}\int_{B_{1/r}\backslash B_{9/10}(z)}\frac{|y|^{\tau}w^p(y)}{|x_2-y|^{n-2\sigma}}dy\\
&\leq \left(\frac{7}{2}\right)^{n-2\sigma}g(x_2).
\end{split}
\end{equation*}
Hence, $g$ satisfies the Harnack inequality in $B_{1/2}(z)$. Since $h\in C^1(\overline{B_1})$ is a positive function, there exist a constant $C_0\geq 1$ such that $\max_{\overline{B_{1/2}(z)}}v\leq C_0\min_{\overline{B_{1/2}(z)}}v$. On the other hand, we can write $w$ as
\[
w(x)=\int_{B_{9/10}(z)}\frac{|y|^{\tau}w^p(y)}{|x-y|^{n-2\sigma}}dy+g(x)+v(x)\quad\quad\mbox{in } \  B_{1/2}(z),
\]
then from Proposition 2.2 in \cite{JLX} we conclude that
\[
\sup_{B_{1/2}(z)}w\leq C\inf_{B_{1/2}(z)}w.
\]
A covering argument leads to
\[
\sup_{B_{3/2}\backslash B_{1/2}}w\leq C\inf_{B_{3/2}\backslash B_{1/2}}w.
\]
We complete the proof of Harnack inequality by rescaling back to $u$.

\end{proof}
\section{Asymptotical radial symmetry}\label{sec:asymptotical}
Last, we give a proof of the Theorem \ref{i3}  for completely.
\subsection{Proof of Theorem \ref{i3}}
\begin{proof}[Proof of Theorem \ref{i3}]
Assume that there exists some positive constant $\va\in (0,1)$ such that for all $0<\lda<|x|\le \va$, $y\in B_{3/2}\backslash (B_\lda(x)\cup\{0\})$,
\be\label{eq:small stop1}
u_{x,\lda}(y)\le u(y),
\ee
where
\[
u_{x,\lda }(y):= \left(\frac{\lda }{|y-x|}\right)^{n-2\sigma}u\left(x+\frac{\lda^2(y-x)}{|y-x|^2}\right).
\]
Let $r>0$ and $x_1$, $x_2\in \partial B_r$ be such that
$$u(x_1)=\max_{\partial B_r} u,~~~~u(x_2)=\min_{\partial B_r} u,$$
and define $$x_3:=x_1+\frac{\varepsilon(x_1-x_2)}{4|x_1-x_2|},\quad
\lambda:=\sqrt{\frac{\varepsilon}{4}\Big(|x_1-x_2|+\frac{\varepsilon}{4}\Big)}.$$
Then
\be\label{PP}
|x_3|=\left|x_1+\frac{\varepsilon(x_1-x_2)}{4|x_1-x_2|}\right|\leq r+\frac{\varepsilon}{4}.
\ee
Via some direct computations and   $|x_1|^2=|x_2|^2=r^2$, we find that
\begin{equation*}
\begin{split}
\lda^2-|x_3|^2&=\frac{\varepsilon}{4}\left(|x_1-x_2|+\frac{\varepsilon}{4}\right)
-\left|x_1+\frac{\varepsilon(x_1-x_2)}{4|x_1-x_2|}\right|^2\\
&=\frac{\varepsilon(|x_2|^2-|x_1|^2)}{4|x_1-x_2|}-x_1^2
=-x_1^2<0,
\end{split}
\end{equation*}
which follows from this and \eqref{PP} that $\lda<|x_3| <\varepsilon$ by choosing $r<\frac{3\varepsilon}{4}$.

It follows from \eqref{eq:small stop1} that
$$u_{x_3,\lda}(x_2)\le u(x_2).$$
Since
\[
x_2-x_3=x_2-x_1+\frac{\varepsilon(x_2-x_1)}{4|x_1-x_2|}
=\frac{x_2-x_1}{|x_1-x_2|}\left(|x_1-x_2|+\frac{\varepsilon}{4}\right),
\]
then
\[
|x_2-x_3|=|x_1-x_2|+\frac{\varepsilon}{4},
\]
\[
\frac{x_2-x_3}{|x_2-x_3|^2}=\frac{x_2-x_1}{|x_1-x_2|\left(|x_1-x_2|+\frac{\varepsilon}{4}\right)},
\]
and
\[
\frac{\lambda^2(x_2-x_3)}{|x_2-x_3|^2}=\frac{\va(x_2-x_1)}{4|x_1-x_2|}.
\]
Hence,
\begin{equation*}
\begin{split}
u_{x_3,\lda}(x_2)&=\left(\frac{\lda}{|x_2-x_3|}\right)^{n-2\sigma}u\left(x_3+\frac{\lambda^2(x_2-x_3)}{|x_2-x_3|^2}\right)\\
&=\left(\frac{\lda}{|x_1-x_2|+\frac{\varepsilon}{4}}\right)^{n-2\sigma}u\left(x_3+\frac{\va(x_2-x_1)}{4|x_1-x_2|}\right)\\
&= \left(\frac{\lda}{|x_1-x_2|+\frac{\varepsilon}{4}}\right)^{n-2\sigma}u(x_1).
\end{split}
\end{equation*}
On the other hand,
\begin{equation*}
\begin{split}
u_{x_3,\lda}(x_2)&= \left(\frac{\lda}{|x_1-x_2|+\frac{\varepsilon}{4}}\right)^{n-2\sigma}u(x_1)
=\frac{u(x_1)}{\left(\frac{4|x_1-x_2|}{\varepsilon}+1\right)^{\frac{n-2\sigma}{2}}}
\geq\frac{u(x_1)}{\left(\frac{8r}{\varepsilon}+1\right)^{\frac{n-2\sigma}{2}}},
\end{split}
\end{equation*}
then$$u(x_1)\leq\left(\frac{8r}{\varepsilon}+1\right)^{\frac{n-2\sigma}{2}}u_{x_3,\lda}(x_2)\leq \left(1+Cr\right)^{\frac{n-2\sigma}{2}}u(x_2),$$
for some $C=C(\va)$. That is,
$$\max_{\partial B_r} u\leq (1+Cr) \min\limits_{\partial B_r} u.$$
Hence for any $x\in \partial B_r$,
$$\frac{u(x)}{\bar u(|x|)}-1\leq \frac{\max_{\partial B_r} u}{\min_{\partial B_r} u}-1\leq Cr,$$
$$\frac{u(x)}{\bar u(|x|)}-1\geq \frac{\min_{\partial B_r} u}{\max_{\partial B_r} u}-1\geq \frac{1}{1+Cr}-1> -Cr,$$
In conclusion, we have
$$\left|\frac{u(x)}{\bar u(|x|)}-1\right|\leq Cr.$$
It follows that
\[
u(x)=\bar u(|x|)(1+O(r))\quad\mbox{as }\ x\rightarrow0.
\]
Therefore, in order to complete the proof of Theorem \ref{i3}, it suffices to prove \eqref{eq:small stop1}.
\end{proof}
\subsection{The proof of \eqref{eq:small stop1}}
Replacing $u(x)$ by $r^{\frac{2\sigma+\tau}{p-1}}u(rx)$ and $h(x)$ by $r^{\frac{2\sigma+\tau}{p-1}}h(rx)$ for $r=\frac{2}{3}$, we can  consider the equation \eqref{i1} in $B_{3/2}$ for convenience, namely,
\be\label{i11}
u(y)=\int_{B_{2/3} }\frac{|z|^{\tau}u^p(z)}{|y-z|^{n-2\sigma}}dz+h(y)\quad\quad\mbox{in }\ \ B_{3/2} \backslash\{0\},
\ee
with $h\in C^1(\overline{B_{3/2}})$ is positive and $|\nabla \ln h|\leq C$ in $\overline{B_{3/2}}$.
Moreover, if we extend $u$ to be identically $0$ outside $B_{3/2}$, then \eqref{i11} can be written as
\[
u(y)=\int_{\R^n}\frac{|z|^{\tau}u^p(z)}{|y-z|^{n-2\sigma}}dz+h(y)\quad\quad\mbox{in }\ \ B_{3/2} \backslash\{0\}.
\]
For all $0<|x|<\frac{1}{16}$ and $\lambda >0$, it is a straightforward computation to show that
\begin{equation*}
u_{x,\lda }(y)=
\int_{\R^n}\left(\frac{\lambda}{|z-x|}\right)^{p^*}\frac{\left|z_{x,\lambda}\right|^{\tau}u_{x,\lambda}^p(z)}{|y-z|^{n-2\sigma}}dz+h_{x,\lambda}(y)\quad\quad\mbox{in }\ \ B^{x,\lambda}_{3/2},
\end{equation*}
where $z_{x,\lambda}:=x+\frac{\lambda^2(z-x)}{|z-x|^2}$, $p^*:=n+2\sigma-p(n-2\sigma)$, $B^{x,\lambda}_{3/2}:=\left\{y_{x,\lambda},y\in B_{3/2}\right\}$.
It follows that
\begin{equation*}
\begin{split}
u(y)-u_{x,\lda }(y)=&\int_{|z-x|\geq \lambda}K(x,\lambda;y,z)\left(|z|^{\tau}u^p(z)-\left(\frac{\lambda}{|z-x|}\right)^{p^*}\left|z_{x,\lambda}\right|^{\tau}u_{x,\lambda}^p(z)\right)\\
&+h(y)-h_{x,\lda }(y),
\end{split}
\end{equation*}
where
\[
K(x,\lambda;y,z):=\frac{1}{|y-z|^{n-2\sigma}}-\left(\frac{\lambda}{|y-x|}\right)^{n-2\sigma}\frac{1}{|y_{x,\lambda}-z|^{n-2\sigma}}.
\]
On the other hand, since $h\in C^1(\overline{B_{3/2}})$ is positive and $|\nabla \ln h|\leq C$ in $B_{3/2}$, then by \cite[Lemma 3.1]{JX}, there exists $r_0\in (0,1/2)$ depending only on $n$, $\sigma$ and $C$ such that for every $x\in B_1$ and $0<\lambda\leq r_0$, there holds
\be\label{j3}
h_{x,\lda }(y)\leq h(y) \quad\quad\mbox{in} \ \ B_{3/2}.
\ee
The aim is to show that there exists some positive constant $\va\in (0,r_0)$ such that for $|x|\le \va$, $\lda\in(0,|x|)$,
\be\label{eq:small stop}
u_{x,\lda}(y)\le u(y) \quad \mbox{in }\  B_{3/2}\backslash (B_\lda(x)\cup\{0\}),
\ee
that is \eqref{eq:small stop1}.

\subsection{The proof of \eqref{eq:small stop}}
To prove \eqref{eq:small stop}, for fixed $x\in B_{1/16}\backslash \{0\}$, we first define
\[
\bar \lda(x):=\sup \left\{0<\mu\le |x|\ \big|\ u_{x,\lda}(y)\leq u(y)\ \mbox{in }\ B_{3/2}\backslash (B_\lda(x)\cup\{0\}),~\forall~ 0<\lda <\mu\right\},
\]
and then show $\bar \lda(x)=|x|$.

For sake of clarity, the proof of  \eqref{eq:small stop} is divided into  three steps.  For the first step, we need the following Claim 1 to make sure that $\bar \lda(x)$ is well defined.

\textbf{Claim 1}: There exists $\lda_0(x)<|x|$ such that for all $\lda\in(0,\lda_0(x))$,
$$ u_{x,\lda}(y)\leq u(y) \quad\mbox{ in }\ B_{3/2}\backslash (B_\lda(x)\cup\{0\}).$$
Second, we give that

\textbf{Claim 2}: There exists a positive constant $\va\in(0,r_0)$ sufficiently small such that for all $|x|\le \va$, $\lda\in(0,|x|)$,
\[
u_{x,\lda}(y)<u(y) \quad \mbox{ in }\  B_{3/2}\backslash B_{1/4}.
\]
Last,  we are going to prove that

\textbf{Claim 3}:  $$\bar \lda(x)=|x|.$$

\begin{proof}[Proof of Claim 1] First of all, we are going to show that there exist $\mu$ and $\lambda_0(x)$ satisfying $0<\lambda_0(x)<\mu<|x|$ such that for all $\lda \in (0,\lambda_0(x))$,
\be\label{Q}
u_{x,\lda }(y)\leq u(y) \quad \mbox{ in } \ \overline{B_\mu(x)}\backslash B_\lda(x).
\ee
Then we will prove that for all $\lda \in (0,\lda_0(x))$,
\be\label{N}
u_{x,\lambda}(y)\leq u(y) \quad \mbox{ in }\ B_{3/2}\backslash\left(\overline{B_\mu(x)}\cup\{0\}\right).
\ee
Indeed, for every $0<\lda<\mu<\frac{1}{2}|x|$, we have
\[
|\nabla \ln u|\leq C_0 \quad\quad\mbox{in} \  \overline{B_{|x|/2}(x)}.
\]
Then for all $0<r<\mu:=\min\left\{\frac{|x|}{4}, \frac{n-2\sigma}{2C_0}\right\}$,  $\theta\in S^{n-1}$,
\begin{equation*}
\begin{split}
\frac{d}{dr}\left(r^{\frac{n-2\sigma}{2}}u(x+r\theta)\right)
&=r^{\frac{n-2\sigma}{2}-1}u(x+r\theta)\left(\frac{n-2\sigma}{2}-r\frac{\nabla u\cdot \theta}{u}\right)\\
&\geq r^{\frac{n-2\sigma}{2}-1}u(x+r\theta)\left(\frac{n-2\sigma}{2}-C_0r\right)
>0.
\end{split}
\end{equation*}
For any $y\in B_{\mu}(x)$, $0<\lambda<|y-x|\leq \mu$, let
\[
\theta=\frac{y-x}{|y-x|},\quad r_1=|y-x|,\quad r_2=\frac{\lambda^2}{|y-x|^2}r_1.
\]
It follows that
\[
r_2^{\frac{n-2\sigma}{2}}u(x+r_2\theta)<r_1^{\frac{n-2\sigma}{2}}u(x+r_1\theta).
\]
That is \eqref{Q}. By equation \eqref{i1},  we have
\be\label{i4}
u(x)\geq 4^{2\sigma-n}\int_{B_{3/2}}|y|^{\tau}u^p(y)dy=:C_1>0,
\ee
and thus we can find $0<\lambda_0(x)\ll\mu$ such that, for every $\lambda\in(0,\lambda_0(x))$,
\[
u_{x,\lambda}(y)\leq u(y) \quad \mbox{ in }\ B_{3/2}\backslash\left(\overline{B_\mu(x)}\cup\{0\}\right),
\]
that is \eqref{N}.
\end{proof}

\begin{proof}[Proof of Claim 2] For $\frac{1}{4}\le |y|\le \frac{3}{2}$ and $0<\lda<|x|<\frac 18$, we have
\[
|y-x|\geq |y|-|x|\geq \frac 18>|x|.
\]
Hence
\[
\left|x+\frac{\lda^2(y-x)}{|y-x|^2}\right|\le |x|+\frac{|x|^2}{|y-x|}\le 2|x|,
\]
and
\[
\left|x+\frac{\lda^2(y-x)}{|y-x|^2}\right|\geq |x|- \frac{|x|^2}{|y-x|}\geq  \frac{|x|}{2}.
\]
It follows from Theorem \ref{i2} that
$$u\left(x+\frac{\lda^2(y-x)}{|y-x|^2}\right)\leq  C|x|^{-\frac{2\sigma+\tau}{p-1}},$$
Thus, for $0<\lda<|x|<\frac{1}{8},\ \frac{1}{4}\le|y|\le \frac{3}{2}$, we conclude that
\begin{equation}\label{i5}
\begin{split}
u_{x,\lda}(y)&\leq \left(\frac{\lda}{|y-x|}\right)^{n-2\sigma}C|x|^{-\frac{2\sigma+\tau}{p-1}}\\
&\leq C\lda^{n-2\sigma}|x|^{-\frac{2\sigma+\tau}{p-1}}\\
&\leq C|x|^{\frac{p(n-2\sigma)-n-\tau}{p-1}}\leq C|\varepsilon|^{\frac{p(n-2\sigma)-n-\tau}{p-1}}.
\end{split}
\end{equation}
Since $\frac{n+\tau}{n-2\sigma}<p\leq\frac{n+2\sigma+2\tau}{n-2\sigma}$, we have $\frac{p(n-2\sigma)-n-\tau}{p-1}>0$. Then by \eqref{i4}, $\va>0$ can be chosen sufficiently small to guarantee that for all $0<\lda<|x|\le \va<r_0$ and $\frac{1}{4}\le |y|\le \frac{3}{2}$,
\be\label{li}
u_{x,\lda}(y)\le C|x|^{\frac{p(n-2\sigma)-n-\tau}{p-1}}<u(y).
\ee
\end{proof}
\begin{proof}[Proof of Claim 3]
We prove Claim 3 by contradiction. Assume $\bar \lda(x)<|x|\le \va<r_0$ for some $x\neq 0$. We want to show that there exists a positive constant $\widetilde{\va}\in \left(0,\frac{|x|-\bar\lda(x)}{2}\right)$ such that for $\lda\in(\bar\lda(x),\bar\lda(x)+\widetilde{\va})$,
\be\label{H}
u_{x,\lda}(y)\leq u(y) \quad \mbox{ in }\ B_{3/2}\backslash (B_\lda(x)\cup\{0\}),
\ee
which contradicts  the definition of $\bar\lda(x)$, then we obtain $\bar\lda(x)=|x|$.

By the Claim 2, it is obviously to obtain that  \eqref{H} in $B_{3/2}\backslash B_{1/4}$. Next, we need to consider the region $B_{1/4}\backslash (B_\lda(x)\cup\{0\})$.

It is a straightforward computation to show that for every $\bar \lambda(x)\leq\lambda< |x|\leq r_0$,
\begin{equation*}
\begin{split}
u(y)-u_{x,\lda }(y)\geq&\int_{B_{1/2}\backslash B_{\lambda}(x)}K(x,\lambda;y,z)\left(|z|^{\tau}u^p(z)-\left(\frac{\lambda}{|z-x|}\right)^{p^*}\left|z_{x,\lambda}\right|^{\tau}u_{x,\lambda}^p(z)\right)\\
&+J(x,\lambda,u,y),
\end{split}
\end{equation*}
where \eqref{j3} is used in the above inequality and
\begin{equation*}
\begin{split}
J(x,\lambda,u,y):=&\int_{B_{3/2}\backslash B_{1/2}}K(x,\lambda;y,z)\left(|z|^{\tau}u^p(z)-\left(\frac{\lambda}{|z-x|}\right)^{p^*}\left|z_{x,\lambda}\right|^{\tau}u_{x,\lambda}^p(z)\right)dz\\
&-\int_{B_{3/2}^c}K(x,\lambda;y,z)\left(\frac{\lambda}{|z-x|}\right)^{p^*}\left|z_{x,\lambda}\right|^{\tau}u_{x,\lambda}^p(z)dz.
\end{split}
\end{equation*}
It follows that \cite[Proposition 1.3]{BL},
\begin{equation*}
\begin{split}
J(x,\lambda,u,y)
\geq&\int_{B_{3/2}\backslash B_{1/2}}K(x,\lambda;y,z)|z|^{\tau}\left(u^p(z)-u_{x,\lambda}^p(z)\right)dz\\
&-\int_{B_{3/2}^c}K(x,\lambda;y,z)|z|^{\tau}u_{x,\lambda}^p(z)dz.
\end{split}
\end{equation*}
By \eqref{i4} and \eqref{i5}, we have
\begin{equation*}
\begin{split}
J(x,\lambda,u,y)\geq&\left(\frac{3}{2}\right)^{\tau}\int_{B_{3/2}\backslash B_{1/2}}K(x,\lambda;y,z)\left(C_1^p-\left(C|\varepsilon|^{\frac{p(n-2\sigma)-n-\tau}{p-1}}\right)^p\right)dz\\
&-\left(\frac{3}{2}\right)^{\tau}\int_{B_{3/2}^c}K(x,\lambda;y,z)\left(\left(\frac{|x|}{|z-x|}\right)^{n-2\sigma}|x|^{-\frac{2\sigma+\tau}{p-1}}\right)^pdz.
\end{split}
\end{equation*}
Since $\frac{n+\tau}{n-2\sigma}<p\leq\frac{n+2\sigma+2\tau}{n-2\sigma}$, we have $\frac{p(n-2\sigma)-n-\tau}{p-1}>0$. Then  $\va>0$ can be chosen sufficiently small to guarantee that
\begin{equation*}
\begin{split}
J(x,\lambda,u,y)
\geq&\frac{C_1^p}{2}\left(\frac{3}{2}\right)^{\tau}\int_{B_{3/2}\backslash B_{1/2}}K(x,\lambda;y,z)dz\\
&-\left(\frac{3}{2}\right)^{\tau}|\varepsilon|^{\frac{p(n-2\sigma)-n-\tau}{p-1}}\int_{B_{3/2}^c}K(x,\lambda;y,z)\frac{1}{|z-x|^{p(n-2\sigma)}}dz\\
\geq&\frac{C_1^p}{2}\left(\frac{3}{2}\right)^{\tau}\int_{B_{23/16\backslash 9/16}}K(0,\lambda;y-x,z)dz\\
&-\left(\frac{3}{2}\right)^{\tau}\left(\frac{16}{7}\right)^{p(n-2\sigma)}|\varepsilon|^{\frac{p(n-2\sigma)-n-\tau}{p-1}}\int_{B_{23/16}^c}K(0,\lambda;y-x,z)dz.
\end{split}
\end{equation*}
Indeed, since  for $|y-x|=\lambda<\frac{1}{16}$,
\[
K(0,\lambda;y-x,z)=0,
\]
and for $|z|\geq\frac{3}{8}$, $|y-x|=\lambda$,
\[
(y-x)\cdot \nabla_y K(0,\lambda;y-x,z)=(n-2\sigma)|y-x|^{2\sigma-n-2}(|z|^2-|y-x|^2)>0.
\]
Using the positive and smoothness of $K$, we have
\be\label{i6}
\frac{\delta_1(|y-x|-\lambda)}{|y-x-z|^{n-2\sigma}}\leq  K(0,\lambda;y-x,z)\leq \frac{\delta_2(|y-x|-\lambda)}{|y-x-z|^{n-2\sigma}},
\ee
for $\bar \lambda(x)\leq\lambda\leq |y-x|\leq |x|+\frac{1}{4}<\frac{5}{16}$, $\frac{3}{8}\leq|z|\leq M<+\infty$, where $M$ and $0<\delta_1<\delta_2<+\infty$ are positive constants. If $M$ is large enough, then
\[
0<c_2\leq (y-x)\cdot \nabla_y (|y-x|^{n-2\sigma}K(0,\lambda;y-x,z))\leq c_3<+\infty.
\]
Thus, \eqref{i6} holds for $|z|\geq M$, $\bar \lambda(x)\leq\lambda\leq |y-x|\leq |x|+\frac{1}{4}$.

With the help of it, for $y\in B_{1/4}\backslash(B_\lambda(x)\cup\{0\})$, there exists positive constants $C_2$ and $C_3$ such that
\begin{equation*}
\begin{split}
J(x,\lambda,u,y)
\geq&\frac{C_1}{2}\left(\frac{3}{2}\right)^{\tau}\int_{B_{23/16\backslash 9/16}}\frac{\delta_1(|y-x|-\lambda)}{|y-x-z|^{n-2\sigma}}dz\\
&-\left(\frac{3}{2}\right)^{\tau}\left(\frac{16}{7}\right)^{p(n-2\sigma)}|\varepsilon|^{\frac{p(n-2\sigma)-n-\tau}{p-1}}\int_{B_{23/16}^c}\frac{\delta_2(|y-x|-\lambda)}{|y-x-z|^{n-2\sigma}}dz\\
\geq& C_2(|y-x|-\lambda)-C_3(|y-x|-\lambda)|\varepsilon|^{\frac{p(n-2\sigma)-n-\tau}{p-1}}.
\end{split}
\end{equation*}
For $\varepsilon$ sufficiently small, we have
\begin{equation*}
J(x,\lambda,u,y)
\geq \frac{C_2}{2}(|y-x|-\lambda).
\end{equation*}
It follows that we can choose $\widetilde{\va}\in \left(0,\frac{|x|-\bar\lda(x)}{2}\right)$ such that for every $\bar \lambda(x)\leq\lambda\leq \bar\lambda(x)+\widetilde{\va}$, and $y\in B_{1/4}\backslash (B_\lambda(x)\cup \{0\})$,
\begin{equation*}
\begin{split}
u(y)-u_{x,\lda }(y)&\geq\int_{B_{1/2}\backslash B_{\lambda}(x)}K(x,\lambda;y,z)\left(|z|^{\tau}u^p(z)-\left(\frac{\lambda}{|z-x|}\right)^{p^*}\left|z_{x,\lambda}\right|^{\tau}u_{x,\lambda}^p(z)\right)dz\\
&\geq \int_{B_{1/2}\backslash B_{\lambda}(x)}K(x,\lambda;y,z)|z|^{\tau}\left(u^p(z)-u_{x,\lambda}^p(z)\right)dz.
\end{split}
\end{equation*}
So Claim 2 gives that
\begin{equation*}
\begin{split}
u(y)-u_{x,\lda }(y)
\geq& \int_{B_{1/4}\backslash B_{\lambda}(x)}K(x,\lambda;y,z)|z|^{\tau}\left(u^p(z)-u_{x,\lambda}^p(z)\right)dz\\
&+ \int_{B_{1/2}\backslash B_{5/16}}K(x,\lambda;y,z)|z|^{\tau}\left(u^p(z)-u_{x,\lambda}^p(z)\right)dz\\
\geq& \int_{B_{1/4}\backslash B_{\lambda}(x)}K(x,\lambda;y,z)|z|^{\tau}\left(u_{x,\bar\lambda(x)}^p(z)-u_{x,\lambda}^p(z)\right)dz\\
&+ 2^{\tau}\int_{B_{1/2}\backslash B_{5/16}}K(x,\lambda;y,z)\left(u^p(z)-u_{x,\lambda}^p(z)\right)dz\\
\geq& -4^{-\tau}\int_{B_{1/4}\backslash B_{\lambda}(x)}K(x,\lambda;y,z)\left|u_{x,\bar\lambda(x)}^p(z)-u_{x,\lambda}^p(z)\right|dz\\
&+ 2^{\tau}\int_{B_{1/2}\backslash B_{5/16}}K(x,\lambda;y,z)\left(u^p(z)-u_{x,\lambda}^p(z)\right)dz.
\end{split}
\end{equation*}
Since $\|u\|_{C(B_{\bar\lambda(x)+\widetilde{\va}}(x))}\leq C$, it follows that there exists some constant $C>0$ such that for any $\bar\lambda(x)\leq\lambda\leq\bar\lambda(x)+\widetilde{\va}$, $z\in B_{1/4}\backslash B_{\lambda}(x)$,
\[
|u_{x,\bar\lambda}^p(z)-u_{x,\lambda}^p(z)|\leq C(\lambda-\bar\lambda(x))\leq C\widetilde{\va}.
\]
Moreover, for $z\in \overline{B_{1/2}}\backslash B_{5/16}$, there exists some constant $C_1>0$ such that
\[
u^p(z)-u_{x,\lambda}^p(z)\geq C_1.
\]
Hence, we have
\begin{equation*}
\begin{split}
u(y)-u_{x,\lda}(y)
&\geq -C\widetilde{\va}\int_{B_{1/4}\backslash B_{\lambda}(x)}K(x,\lambda;y,z)dz+ C_1\int_{B_{1/2}\backslash B_{5/16}}K(x,\lambda;y,z)dz\\
&\geq -C\widetilde{\va}\int_{B_{1/4}\backslash B_{\lambda}(x)}K(x,\lambda;y,z)dz+ C_1\int_{B_{7/16}\backslash B_{3/8}}K(0,\lambda;y-x,z)dz.
\end{split}
\end{equation*}
On the other hand, since
\begin{equation*}
\begin{split}
\int_{B_{1/4}\backslash B_{\lambda}(x)}K(x,\lambda;y,z)dz&\leq
\int_{B_{5/16}\backslash B_{\lambda}}K(0,\lambda;y-x,z)dz\\
&\leq C(|y-x|-\lambda),
\end{split}
\end{equation*}
and
\[
\int_{B_{7/16}\backslash B_{3/8}}K(0,\lambda;y-x,z)dz\geq \frac{\delta_1(|y-x|-\lambda)}{|y-x-z|^{n-2\sigma}}.
\]
Then we can choose $\widetilde{\va}$ sufficient small such that for $\bar\lambda(x)\leq \lambda\leq \bar\lambda(x)+\widetilde{\va}$,
\[
u_{x,\lda}(y)\leq u(y)\quad\quad\mbox{in } \ \ B_{1/4}\backslash (B_\lambda(x)\cup\{0\}).
\]
Combining  Claim 2, we get a contradiction and then we finish the proof.
\end{proof}

{\bf Acknowledgements}
We would like to express our deep thanks to Professor Jiguang Bao and Professor  Jingang Xiong for useful discussions on the subject of this paper.

{}
\bigskip

\medskip

\noindent Y. Li

\noindent Beijing Jiaotong University,\\
 Beijing 100044, People's Republic of China\\[1mm]
\noindent {\it E-mails}: \texttt{lyimei@bjtu.edu.cn}
%
%

\end{document}